\documentclass[a4paper,american,11pt, reqno, final]{amsart}

\usepackage[utf8]{inputenc}
\RequirePackage[pdfusetitle, hidelinks]{hyperref}
\RequirePackage[capitalize, noabbrev]{cleveref}

\usepackage{amsmath}
\usepackage{amssymb}
\usepackage{hyperref}
\usepackage{babel}
\usepackage{tikz}
\usepackage{pgf,tikz,pgfplots}
\usepackage{graphicx}
\usepackage[disable]{todonotes}
\usepackage{cleveref}
\usepackage{comment}
\usepackage{array}

\usepackage{thmtools}
\usepackage{thm-restate}

\usetikzlibrary{arrows}

\usetikzlibrary{calc}
\usepackage{amsfonts}
\input xy
\xyoption{all}

\newtheorem{thm}{Theorem}[section]

\newtheorem{definition}[thm]{Definition}

\newtheorem{proposition}[thm]{Proposition}

\newtheorem{remark}[thm]{Remark}

\newtheorem{conj}[thm]{Conjecture}
          {\theoremstyle{definition}
}
          {\theoremstyle{definition}

\newtheorem{example}[thm]{Example}
}

\newtheorem{ques}[thm]{Question}
 \numberwithin{equation}{section}

\pgfplotsset{compat=1.18} 

\begin{document}

\title{Chern Numbers of Matroids}
\author[]{Eline Mannino}

\maketitle
\begin{abstract}
We define \textit{Chern numbers} of a matroid. These numbers are obtained when intersecting appropriate matroid Chern-Schwartz-MacPherson cycles defined by L\'opez de Medrano, Rinc\'on, and Shaw. We prove that when a matroid arises from a complex hyperplane arrangement the Chern numbers of the matroid correspond to the Chern numbers of the log cotangent bundle. 

A matroid $M$ of rank 3 has two Chern numbers. We prove that they are positive and that their ratio is bounded by 3, which is analogous to the \textit{Bogomolov-Miyaoka-Yau inequality}. If the matroid is orientable, we generalize a result of Eterovi\'c, Figuera, and Urz\'ua to prove that the ratio is bounded above by 5/2. Finally, we give a formula for the Chern numbers of the uniform matroid of any rank.
\end{abstract}

\section{Introduction}
\label{sec:intro}

We introduce new matroid invariants, namely their \textit{Chern numbers}. These are numbers obtained when intersecting the  \textit{Chern Schwartz MacPherson cycles} of a matroid (CSM cycles) introduced by 
L\'opez de Medrano, Rinc\'on, and Shaw  in \cite{de2020chern}.

More specifically, to a matroid $M$ of rank $d+1$ Ardila and Klivans assign a rational polyhedral fan $\Sigma_M$ called the \textit{Bergman fan} of $M$ (\cref{def: Bergman}) having a cone for each chain of flats of $M$. Denote the \textit{Minkowski weights} supported on $\Sigma_M$ by $MW_{\ast}(\Sigma_M)$, where the grading is by dimension. Then $MW_{\ast}(\Sigma_M)$ has the structure of a graded ring where the product 
$$MW_{k}(\Sigma_M) \times MW_{l}(\Sigma_M) \to MW_{d- k - l}(\Sigma_M)$$ 
can be defined in several ways, see for example \cite{shaw2013tropical}, \cite{franccois2013diagonal}, and \cite{adiprasito2018hodge}. Moreover, there exists an isomorphism $w: MW_0(\Sigma_M)\rightarrow \mathbb{Z}$ naturally defined by the weight assigned to the vertex of the fan.

The CSM cycles of a matroid are collections of Minkowski weights $\text{csm}_k(M)$ in $MW_k(M)$ 
for all $k = 0, \dots, d$. The weights 
are determined by the beta invariants of matroid minors. Our  definition in this text is the following. 

\begin{restatable}[]{definition}{cnumbers}
\label{def: cnumbers}
Let $M$ be a rank $d+1$ matroid. The \textit{Chern numbers of a matroid} $\bar{c}_1^{k_1}\bar{c}_2^{k_2}\cdots \bar{c}_d^{k_d}(M)$ are 
$$\bar{c}_1^{k_1}\bar{c}_2^{k_2}\cdots \bar{c}_d^{k_d}(M)= \text{w}(\textnormal{csm}_{d-1}^{k_1}(M)\textnormal{csm}_{d-2}^{k_2}(M)\cdots \textnormal{csm}_0^{k_d}(M)),$$  where $\sum_{i=1}^d i k_i=d$, and the map $\text{w}$ is the map returning the multiplicity associated to the vertex.
\end{restatable}

In algebraic geometry, the \textit{Chern-Schwartz-MacPherson classes} generalize the Chern classes of the tangent bundle for varieties that are not smooth and complete see \cite{macpherson1974chern} and \cite{AST_1981__82-83__93_0}. 
In \cite{de2020chern}, the authors prove that the CSM cycles of a matroid representable over the complex numbers are related to CSM classes in algebraic geometry. Namely, when the matroid $M_\mathcal{A}$ arises from an arrangement $\mathcal{A}$ defined over $\mathbb{C}$, the authors prove that 
$$\text{CSM}(1_{C(\mathcal{A})})=\sum_{k=0}^d \text{csm}_k(M_\mathcal{A})\in MW_*(\Sigma_{\mathcal{A}}),$$
where $\text{CSM}(1_{C(\mathcal{A})})$ is the CSM class of the group of constructible function on the complement of the arrangement $C(\mathcal{A})$. Moreover, Aluffi proves that $\text{CSM}(1_{C(\mathcal{A})})$ equals the Chern class of the sheaf of the log cotangent bundle, see \cite{aluffi2012chern}. This allows us to to relate the Chern numbers of matroids representable over $\mathbb{C}$ to Chern numbers in algebraic geometry. 
\begin{restatable}{thm}{main}
\label{main}
Let $\mathcal{A}\subset \mathbb{P}^d_\mathbb{C}$ be a hyperplane arrangement and let $M_\mathcal{A}$ be the corresponding matroid. Let $W_\mathcal{A}$ be the maximal wonderful compactification of its complement $C(\mathcal{A})$ and let $D= W_\mathcal{A} \setminus C(\mathcal{A})$. Let $\Omega^1_{W_\mathcal{A}}(\log D)$ denote the bundle of log differentials on $W_\mathcal{A}$ with poles along $D$ and $\Omega^1_{W_\mathcal{A}}(\log D)^*$ denote its dual. Then 
$$c_1(\Omega^1_{W_\mathcal{A}}(\log D)^*)^{k_1}\cdots c_d(\Omega^1_{W_\mathcal{A}}(\log D)^*)^{k_d}=\bar{c}_1^{k_1}\cdots \bar{c}_d^{k_d}(M_\mathcal{A}).$$
\end{restatable}

Next, for a few particular matroids we find an explicit formula for the Chern numbers. For example, in \cref{prop: ch uniform}, we give a formula for computing the Chern numbers of the uniform matroid. 
In \cref{prop: Chern3} we give a formula for the Chern numbers of simple matroids of rank 3 in terms of the flats. Note that when a matroid is representable over some field $k$, it arises from a hyperplane arrangement over $k$. In fact the formula given in \cref{prop: Chern3} corresponds to Eterovi\'c, Figuera, and Urz\'ua definition for the Chern numbers of line arrangements given in terms of the incidence structure \cite{eterovic2022geography}. 

Inspired by the study of the geography of manifolds, see for example  \cite{hunt1989complex}, we then launch a study of the possible Chern numbers of matroids. 
Focusing specifically on rank $3$, we consider the following question. 

\begin{ques}[Geography of Chern numbers in rank 3]
\label{question: range}
For which pairs $(a,b)\in \mathbb{Z}_{\geq 0}\times \mathbb{Z}_{\geq 0}$ does there exist a rank $3$ matroid such that $$(\overline{c}_1^2(M), \overline{c}_2(M))=(a,b)?$$ 
\end{ques}

We then establish some first results 
in the geography problem for matroids. 

\begin{restatable}{proposition}{newbounds}
\label{prop: new bounds}
Let $M$ be a simple matroid of rank $3$ on a ground set $E$ of size $n$. The Chern numbers $\overline{c}_1^2(M)$ and $\overline{c}_2(M)$ are bounded by
\begin{align*}
     0\leq \overline{c}_1^2(M)&\leq \overline{c}_1^2(U_{3,n}),\ \textup{and} \\
    0\leq \overline{c}_2(M)&\leq\overline{c}_2(U_{3,n}).
\end{align*}
\end{restatable}

We also prove upper and lower bounds on the quotient of the two Chern numbers of rank $3$ matroids. The next theorem is a generalization of Proposition 3.4 in \cite{eterovic2022geography} which proves the same bounds for the quotient of Chern numbers of line arrangements.

\begin{restatable}{thm}{boundsc}
\label{thm: bounds ch 3}
Let $M$ be a simple matroid of rank $3$ on a ground set $E$ of size n with no coloops. Then,
\begin{align*}
    \frac{2n-6}{n-2}\leq \frac{\overline{c}_1^2(M)}{\overline{c}_2(M)}\leq 3.
\end{align*}
Left equality holds if and only if M is the uniform matroid $U_{3,n}$, and right equality holds if and only if M is the matroid of a finite projective plane.
\end{restatable}
Curiously, the right inequality is the same inequality for the quotient of the Chern numbers of complex algebraic surfaces given some restrictions, which is known as the \textit{Bogomolov-Miyaoka-Yau inequality}. Note that, also for the case of representable matroids, the inequality in \cref{thm: bounds ch 3} is a different inequality than the Bogomolov-Miyaoka-Yau inequality. While we are calculating Chern classes of bundles of log differentials (see \cref{main}), they are calculating Chern classes of the complex tangent bundles. Later on, Miyaoka stated an analogy of the Bogomolov-Miyaoka-Yau inequality in a more general context that can be restricted to our case for $X=W_\mathcal{A}$, see Corollary 1.2 in \cite{miyaoka1984maximal}. 

Line arrangements appear in Hirzebruch's study of the Bogomolov-Miyaoka-Yau inequality. To a line arrangement $\mathcal{A}\in \mathbb{P}^2_\mathbb{C}$ and a $n\geq 2$, Hirzebruch associates a smooth algebraic surface $Y$ by desingularizing a covering of $\mathbb{P}^2_\mathbb{C}$ of degree depending on $n$, see \cite{hirzebruch1983arrangements} for details. Hirzebruch proves that the right inequality holds for the quotient of the Chern numbers of $Y$, see the theorem in Section 2.5 of \cite{hirzebruch1983arrangements}. Moreover, Sommes proves that the left inequality holds for surfaces constructed in the same manner as Hirzebruch, see Theorem 5.1 in \cite{sommese1984density}. 

While Hirzebruch  gives, for a few specific line arrangements over $\mathbb{C}$, a surface covering the line arrangement having Chern numbers quotient obtaining 3, we prove that the upper bound of $3$ in Theorem \ref{thm: bounds ch 3} is obtained only for matroids arising from finite projective planes.

Before introducing the next theorem, we need some background. A simple closed curve $L$ in $\mathbb{P}^2_\mathbb{R}$ is called a pseudoline if $\mathbb{P}^2\setminus L$  has one connected component. Moreover, Folkman and Lawrence have proved that there is a one to one correspondence between arrangements of pseudolines and oriented matroids \cite{folkman1978oriented}. The next theorem about the quotient of the Chern numbers is a generalization of Theorem 3.5 in \cite{eterovic2022geography} which is stated in terms of real line arrangements. 
\begin{restatable}{thm}{orientable}
\label{thm: orientable}
   If $M$ is a simple oriented matroid of rank 3 with no coloops, then 
\begin{align*}
    \frac{\overline{c}_1^2(M)}{\overline{c}_2(M)}\leq \frac{5}{2}.
\end{align*}
 Equality is achieved if and only if the pseudoline arrangement of M is simplicial. 
\end{restatable}
\textcolor{red}{}

In the following question  we further relate the study of Chern numbers of matroids to geometry.
\begin{ques}[Representability of Chern numbers in rank 3]
\label{question: realizability 3}
Given any simple matroid $M$ of rank 3 with Chern numbers $(\overline{c}_1^2(M),\overline{c}_2(M))$, does there exist a representable matroid $M_\mathcal{A}$ over some field $k$ of rank $3$ such that $$(\overline{c}_1^2(M_\mathcal{A}), \overline{c}_2(M_\mathcal{A}))=(\overline{c}_1^2(M), \overline{c}_2(M))?$$ 
\end{ques}
Finally, we want to point out some possible implications of answers to this question. In \cref{thm: bounds ch 3} we will see that for a simple matroid of rank 3 the quotient $\overline{c}_1^2(M)/\overline{c}_2(M)=3$ if and only if $M$ is the matroid of a finite projective plane. Moreover, in \cref{ex: ch PG(2 q)} we show that the finite projective plane matroid $PG(2,q)$ of order $q$ has Chern numbers 
\begin{align*}
  \overline{c}_1^2(PG(2,q))=&3(q^3-q^2-q+1),\\
    \overline{c}_2(PG(2,q))=&q^3-q^2-q+1.
\end{align*}
Hence, finding out for which values of $q$ there exists a matroid $M$ having Chern numbers $(\overline{c}_1^2(M),\overline{c}_2(M))=(3(q^3-q^2-q+1),q^3-q^2-q+1)$, and if so if the matroid is representable over a field, would in fact solve the prime power conjecture for finite projective planes (\cite{veblen1906finite}). 

\section*{Acknowledgements}
{
This paper is based on my master thesis, see \cite{mannino2022chern}. Firstly, I want to thank my supervisor Kris Shaw who inspired and encouraged me. I also want to thank Felipe Rinc\'on and John Christian Ottem for carefully reading and commenting on my thesis. I want to thank Edvard Aksnes and Jon Pål Hamre for insightful discussions. I want to thank Ragni Piene for noticing a significant mistake in the bibliography. Finally, I want to thank Matt Wendell Larson for pointing me to the more general case of the the Bogomolov-Miyaoka-Yau inequality and Piotr Pokora for pointing me to previous results regarding pseudoline arrangements. 

I acknowledge the support of the Centre for Advanced Study (CAS) in Oslo, Norway, which funded and hosted the Young CAS research project Real Structures in Discrete, Algebraic, Symplectic, and Tropical Geometries during the 2021/2022 and 2022/2023 academic years. This research was also supported by the Trond Mohn Foundation project “Algebraic and Topological Cycles in Complex and Tropical Geometries".
}
\section{Chern numbers of matroids}
We begin by introducing the necessary background. 
\begin{definition}
\label{def: matroid}
A \textit{matroid} on a finite set $E$ of rank $d+1$ is a function  $$r:2^E\rightarrow \mathbb{Z}$$ 
satisfying
\begin{enumerate}
    \item[1.] $0 \leq r(S) \leq |S|$,
    \item[2.] $S\subseteq U$ implies $r(S)\leq r(U),$
    \item[3.] $r(S\cup U) + r (S\cap U)\leq r(S)+r(U)$\text{ and}
    \item[4.] $r(E)=d+1.$
\end{enumerate}
\end{definition}
\begin{example}
\label{ex: hypArr}
Let  $L\subseteq \mathbb{P}^n_k$ be a linear subspace over a field $k$. And  let
$$\mathcal{A}=\{H_i=L \cap V(x_i)\ |\ 0\leq i \leq n-1\}\subseteq L$$
be the hyperplane arrangement given by intersecting $L$ with the coordinate axis $V(x_i)\subseteq \mathbb{P}^n_k$. We can define a function $r$ on the power set of $\mathcal{A}$: given a subset $S\subseteq \mathcal{A}$ let  $$r(S)= \text{codimension}\Big(\bigcap_{i\in S}H_i\Big)$$
in $L$. The function $r$ is a matroid.
\qed
\end{example}

Given a hyperplane arrangement $\mathcal{A}$ we denote by $M_\mathcal{A}$ the corresponding matroid.

A subset $S \subseteq E$ such that for any element $j$ of $E$ not in $S$ the rank $$r(S\cup \{j\})>r(S)$$ is called a \textit{flat}. The flats of the matroid form a lattice under inclusion called the \textit{lattice of flats} of the matroid. 
A \textit{loop} of a matroid is an element $i\in E$ such that $r(i)=0$. A pair of \textit{parallel points} consists of points $i,j\in E$ such that $r(i,j)=1$. In most of our results, we  assume that the matroid is \textit{simple}, i.e., having neither loops nor parallel points. 

\begin{example}
\label{ex: uniform}
The \emph{uniform matroid} $U_{d+1,n}$ of rank $d+1$ on the ground set $E$ of size $n$ is defined to be the rank function $r$ such that for $S\subseteq E$:
\begin{equation}
    \label{eq: rankU}
    r(S) := \begin{cases} |S|\quad \text{if }|S|<d+1\\
    d+1\quad \text{otherwise}.
    \end{cases}
\end{equation}
Hence, for the uniform matroid any subset $S$ of size $|S|\leq d$ is a flat. 
\qed
\end{example}

The lattice of flats of a matroid $M$ can be represented by a polyhedral fan called its \textit{Bergman fan}, introduced in \cite{ardila2006bergman}, spanned by rays corresponding to the flats of the matroid. More precisely, let the vectors $e_i$ for $i\in E$ be the standard basis of $\mathbb{Z}^E$, and let $N$ be the quotient space $\mathbb{Z}^E/\mathbb{Z}\mathbf{1}$ where the vector $\mathbf{1}=\sum_{i\in E} e_i\in \mathbb{Z}^E$.
Moreover, let $u_i$ be the image of $e_i$ in $N$, and let the vector $u_S = \sum_{i\in S} u_i$ for a subset $S\subset E$.
\begin{definition}[{\cite[Section 3]{ardila2006bergman}}]
\label{def: Bergman}
Let $M$ be a loopless matroid of rank $r=d+1$ on a ground set $E$. The Bergman fan $\sum_M$ is the pure $d$-dimensional polyhedral fan in $N_\mathbb{R}:=N \otimes \mathbb{R}$ consisting of the cones 
$$\sigma_{\mathcal{F}}:=\text{Cone}(u_{F_1},u_{F_2},\cdots, u_{F_k})\subset N_{\mathbb{R}}$$
for each chain of flats $\mathcal{F}:\emptyset\subsetneq F_1\subsetneq \cdots \subsetneq F_k \subsetneq E$ in the lattice of flats of $M$.
\end{definition}
\begin{example}
The Bergman fan of the uniform matroid $U_{3,4}$ is a polyhedral fan of dimension 2 living in $\mathbb{Z}^4/\mathbb{Z}\mathbf{1}\otimes\mathbb{R}$, see \cref{fig: bergman fan U 34}. The fan consists of 12 cones, corresponding to the flags $\emptyset\subsetneq\{i\}\subsetneq\{i,j\}\subsetneq E$, for each pair $i\neq j \in E$. 
\begin{figure}
    \centering
    \includegraphics[scale=0.17]{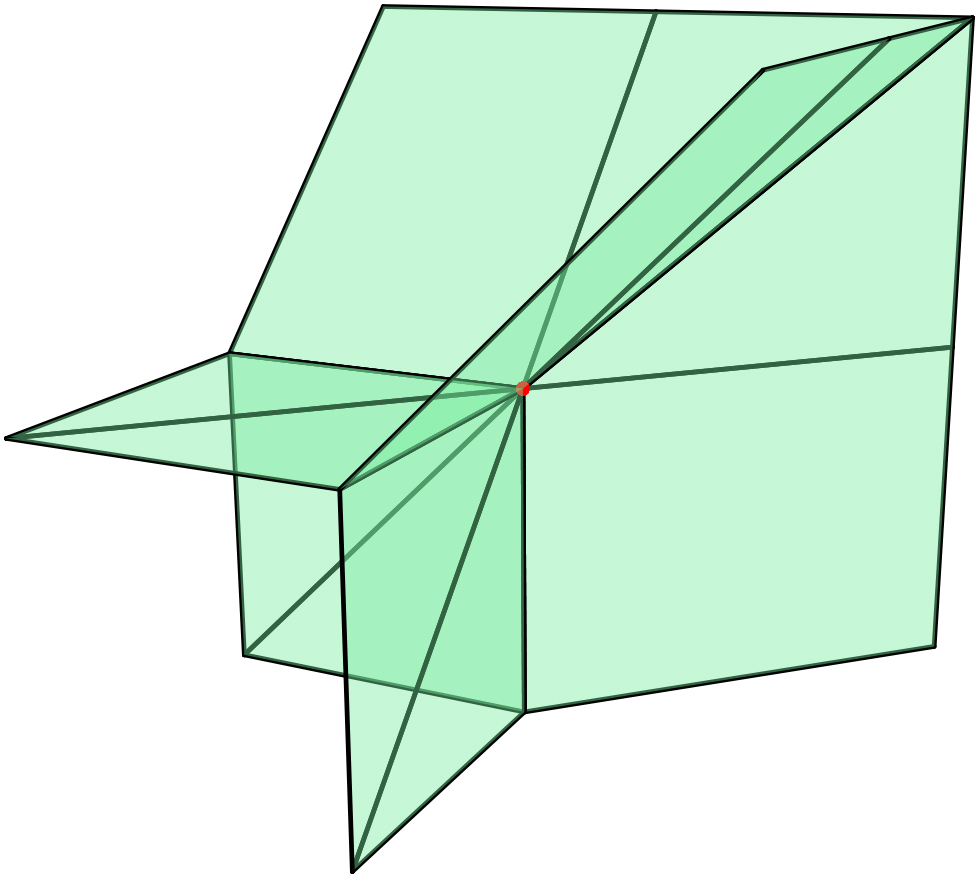}
    \caption{The Bergman fan of $U_{3,4}$}
    \label{fig: bergman fan U 34}
\end{figure}
\qed
\end{example}
Given a subset $S\subseteq E$, the \textit{deletion matroid} $M\setminus S$ is the matroid with ground set $E\setminus S$ and rank function $r_{M\setminus S}$ given by for a subset $U\subseteq E\setminus S$:
$$r_{M\setminus S}(U)= r(U\cup S)-r(S).$$
For a subset $T\subset E$ the \textit{restriction}  matroid $M|T$ is given by $M\setminus (E\setminus T)$. 
Finally, for a subset $X\subseteq E$, the \emph{contraction} $E/X$ is defined to be the matroid on the ground set $E\setminus X$ and rank function $r_{M/X}$ given by for a subset $S\subseteq E\setminus X$:
$$r_{M/X}(S)=r(S\cup X)-r(X).$$
Let $\beta(M)$ be the \textit{beta invariant} of $M$, see \cite{crapo1967higher}.
\begin{definition}[{\cite[Definition 4]{de2020chern}}]
\label{def: csm}
Let $M$ be a loopless matroid, the $k$-dimensional CSM cycle $\text{csm}_k(M)$ of $M$ is a $k$-dimensional Minkowski weight with weigths given by
$$\omega(\sigma_\mathcal{F}) := (-1)^{d-k} \prod^k_{i=0}\beta(M|F_{i+1}/F_i),$$
for a flag $\mathcal{F} := \{\emptyset = F_0 \subsetneq F_1 \subsetneq \cdots \subsetneq F_k \subsetneq F_{k+1} =\{1,...,n\}\}$. The minor $M|F_{i+1}/F_i$ is obtained by restricting to $F_{i+1}$ and contracting to $F_i$.
\end{definition}

\begin{example}
We compute $\text{csm}_k(U_{n,d+1})$. For a cone $\sigma_\mathcal{F}$ corresponding to a flag $\mathcal{F} :=\emptyset =F_0 \subseteq F_1 \subseteq \cdots \subseteq F_k \subseteq F_{k+1} =\{1,...,n\}$ such that $r(F_{i+1})-r(F_i) = 1$ for all flats $F_i$ in the flag, the weight is given by
\begin{align*}
    \omega(\sigma_\mathcal{F}) &= (-1)^{d-k} \binom{n-k-2}{d-k}. 
\end{align*}
Otherwise, the weight of the cone is zero. For example, for a top dimensional cone $\sigma_\mathcal{F}$ the weight is given by: 
\begin{align*}
    \omega(\sigma_\mathcal{F})=1.
\end{align*}
See Example 4. of  \cite{de2020chern} for the details.
{\qed}
\end{example}

We recall our main definition.
\cnumbers*
\begin{example}
\label{ex: beta}
Let $M$ be a rank $d+1$ matroid. The Chern number 
$$\overline{c}_d(M)=(-1)^d\beta(M),$$
which is simply the weight of the vertex of the Bergman fan of the matroid $M$.
\qed
\end{example}
\begin{proposition} 
\label{prop: ch uniform}
The Chern numbers of the uniform matroid $U_{d+1,n}$ are given by 
\begin{align*}
     \bar{c}_1^{k_1}\cdots \bar{c}_d^{k_d}(U_{d+1,n})=(-1)^d\prod_{i=1}^d\binom{n-(d-i)-2}{i}^{k_i}.
    \end{align*}
\end{proposition}
\begin{proof} 

Let the fan $C\in MW_*(\Sigma_{U_{d,n}})$ be the Bergman fan of $U_{d,n}$ with weight 1 on all its top dimensional cones. Then, by Example 15 in \cite{de2020chern} and \cref{ex: beta}, 
 the fan $C^k\in MW_*(\Sigma_{U_{d-k+1,n}})$ is the Bergman fan of $U_{d-k+1,n}$ with weight 1 on all its top dimensional cones. Hence, the Chern numbers of $U_{d+1,n}$ are given by
\begin{align*}
    \bar{c}_1^{k_1}\cdots \bar{c}_d^{k_d}(U_{d+1,n})=&\text{wt}_0\Bigg(\prod_{i=1}^d(-1)^{k_i(d-(d-i))}\binom{n-(d-i)-2}{d-(d-i)}^{k_i}C^{ik_i}\Bigg)\\
    =&(-1)^d\prod_{i=1}^d\binom{n-(d-i)-2}{i}^{k_i},
\end{align*}
where the last equality follows from the fact that the fan $C^d$ is just the vertex of the fan with weight 1, and from the fact that $\sum_{i=1}^d ik_i=d$.
\end{proof}

When the matroid is representable over $\mathbb{C}$ the Chern numbers of the matroid have in fact a geometric meaning. 
\main*
\begin{proof}
Feichtner and Yuzvinsky have proven that for a representable matroid we have the follwoing isomorphism $A^*(M_\mathcal{A})\simeq A^*(W_\mathcal{A})$ \cite{feichtner2004chow}. Moreover, by Poincar\'e duality for the wonderful compactification, the following map  is an  isomorphism of rings:
\[ [W_\mathcal{A}] \cap  \colon A^{\ast}(W_{\mathcal{A}}) \to A_{d-\ast}(W_{\mathcal{A}}), \]
and it respects the graded structures as indicated. This isomorphism is compatible with the degree maps $A^d(W_{\mathcal{A}}) \to \mathbb{Z}$
and $A_0(W_{\mathcal{A}}) \to \mathbb{Z}$, see Definition 5.9 in \cite{adiprasito2018hodge} for a construction of the degree map. 
\todo[inline]{You could cite AHK here but since this is the alg geo setting it is more standard.}

By Theorem 1.1 of \cite{aluffi2012chern} $$[W_\mathcal{A}] \cap \text{c}(\Omega^1_{W_\mathcal{A}}(\text{log} D)^*)  =\text{CSM}(1_{C(\mathcal{A})}).$$

Therefore,
$$[W_\mathcal{A}] \cap \text{c}_k(\Omega^1_{W_\mathcal{A}}(\text{log} D)^*) = \text{CSM}_{d-k}(1_{C(\mathcal{A})}) \in A_{d-k}(W_{\mathcal{A}}).$$ Furthermore, the isomorphism $A_{\ast}(W_{\mathcal{A}}) \simeq MW_{\ast}(\Sigma_{\mathcal{A}})$ follows by the Poincar\'e duality for the wonderful compactification and the linear duality $$MW_k(\Sigma_\mathcal{A})\simeq Hom_\mathbb{Z}(A^k(W_\mathcal{A}),\mathbb{Z})$$ proven in Proposition 5.6 in \cite{adiprasito2018hodge}. Moreover, by Theorem 3.1 of \cite{de2020chern}, we have $$\text{CSM}(1_{C(\mathcal{A}}))
= 
\sum^d_{k=0}\text{csm}_k(M_\mathcal{A}) \in A_*(W_\mathcal{A})\simeq MW_*(\Sigma_\mathcal{A}).
$$

Finally, combining these two results we have that 
$$[W_\mathcal{A}] \cap \text{c}_k(\Omega^1_{W_\mathcal{A}}(\text{log} D)^*) = \text{csm}_{d-k}(M_\mathcal{A}).$$ Since the cap product with $[W_{\mathcal{A}}]$ induces a ring isomorphism compatible with the degree maps we obtain the equality in the statement of the theorem. 
\end{proof}
\section{Chern numbers of matroids of rank 3}
\label{sec: ch3}
In this section we investigate properties of Chern numbers of matroids of rank 3. We first give a formula for the Chern numbers in terms of the number of flats. 
This formula generalizes the formula of Chern numbers of line arrangements defined and studied in \cite{eterovic2022geography}.
\begin{proposition}
\label{prop: Chern3}
Let $M$ be a simple matroid of rank $3$ on a ground set $E$ of size $n$. Let $t_m$ be the number of flats $F$ of rank $2$ of size $m$, then the Chern numbers of $M$ are 
\begin{align*}
   \overline{c}_1^2(M)=9-5n+\sum_{m\geq2}(3m-4)t_m,\\
   \overline{c}_2(M)=3-2n+\sum_{m\geq 2}(m-1)t_m. 
\end{align*}
\end{proposition}
\begin{proof}
For the first statement, we find that the weight associated to the self intersection of the 1-dimensional csm cycle of the Bergman fan of $M$ is 
\begin{equation}
\label{eq: c_1_squared}
    \overline{c}_1^2(M)=(3-n)^2-\sum_{m\geq 2} (2-m)^2t_m,
\end{equation}
 by Definition 2.16 and Equation 3.2 in \cite{shaw2015tropical}. Moreover, we prove that \cref{eq: c_1_squared} equals the first statement of \cref{prop: Chern3} by proving that
 \begin{equation}
 \label{eq: diff}
     n^2-n = \sum_{m\geq 2}(m^2-m) t_m
 \end{equation}
 by induction on $n$. For $n = 3$ the matroid consists of 3 flats of rank 2 of size 2, hence \cref{eq: diff} holds. 
 Assume now that \cref{eq: diff} holds for any matroid $M$ of size $|E|=n$. Then, adding an element to the matroid $M$ increases the left hand side by $2n$. For the right hand side, we first denote by $M'$ the matroid $M$ with one element $i$ added. Note that the rank 2 flats of $M'$ consist of those not containing $i$, which are also rank 2 flats of $M$, and those containing $i$. Now, let $\mathcal{F}_i$ be the set of rank 2 flats of $M'$ containing $i$. The set $\mathcal{F}_i$ equals the disjoint union $\mathcal{F}_i=\mathcal{F}_1 \cup \mathcal{F}_2$, where $\mathcal{F}_1$ consists of the flats $F\in \mathcal{F}_i$ such that $F\setminus \{i\}$ is a rank 2 flat of $M$, and $\mathcal{F}_2$ consists of the flats $F\in \mathcal{F}_i$ such that $F\setminus \{i\}$ is a flat of rank 1 in $M$, hence flats of size 2. Remark that, by the covering axiom of flats, the size $|\mathcal{F}_2|=n-\sum_{F\in \mathcal{F}_{1}}(|F|-1)$. Then, adding an element $i$ to the matroid $M$ changes the right hand side of \cref{eq: diff} by
\begin{align*}
    \sum_{F\in \mathcal{F}_1}(|F|^2-|F|)-\sum_{F\in \mathcal{F}_1}((|F|-1)^2-(|F|-1))&\\
    +2(n-\sum_{F\in \mathcal{F}_{1}}(|F|-1))&=2n.
\end{align*}

This proves \cref{eq: diff}, hence the first statement in \cref{prop: Chern3}.

Recall that by \cref{ex: beta} the Chern number 
$$\overline{c}_d(M)=(-1)^d\beta(M).$$ Hence, for proving the second statement it suffices to compute $\beta(M)$. The characteristic polynomial is given by
\begin{align*}
    &\chi_{M}(\lambda)= \lambda^3-n\lambda^2+\sum_{m\geq 2}(m-1)t_m\lambda-(1-n+\sum_{m \geq 2}(m-1)t_m), 
\end{align*}
where we have used that for flats of rank 2 we have that
$$\sum_{r(F)=2}\mu(\emptyset, F)=\sum_{m\geq2}(m-1)t_m.$$
Then, the beta invariant is given by 
\begin{align*}
    \overline{\chi_{M}}(1)=3-2n+\sum_{m\geq2}(m-1)t_m,
\end{align*}
which equals the Chern number $\overline{c}_2(M)$.
\end{proof}

\begin{remark}
Recall that from a line arrangement, which is uniquely defined by a surface $Y$, we obtain a matroid via \cref{ex: hypArr}. Recall that in \cite{eterovic2022geography} the authors give a combinatorial definition of Chern numbers of line arrangements in terms of the incidence structure. Their definition is inspired from the work of Hirzebruch and they claim that these numbers are the Chern numbers of the corresponding log surface, see their Remark 3.2. Therefore, in the case of representable matroids \cref{prop: Chern3} follows from \cref{main} and Remark 3.2 of \cite{eterovic2022geography}. 
\end{remark}

Matroids arising from finite projective planes $PG(2,q)$ have particularly nice properties.
\begin{example}
\label{ex: ch PG(2 q)}
The finite projective plane matroid $PG(2,q)$ is a matroid on a ground set $E$ of size $q^2+q+1$ having $q^2+q+1$ flats of rank 2 of size $q+1$. Hence, its Chern numbers are given by 
\begin{align*}
    \overline{c}_1^2(PG(2,q))&=3(q^3-q^2-q+1),\\
    \overline{c}_2(PG(2,q))&=q^3-q^2-q+1. 
\end{align*}

For example, the Fano plane $PG(2,2)$ has Chern numbers 
\begin{align*}
    \overline{c}_1^2(PG(2,2))&=9,\\
    \overline{c}_2(PG(2,2))&=3.
\end{align*}
{\qed}
\end{example}
\begin{example}
 The \textit{non-Pappus matroid} is obtained from the pseudoline arrangement in \cref{fig: non-Pappus}. We compute the Chern numbers of the non-Pappus matroid: $(\overline{c}_1^2(M),\overline{c}_2(M))=(28,13)$. 
Note, that the non-Pappus matroid violates Pappus's Theorem, hence it is not representable over any field. However, there exists a representable matroid having the same Chern numbers which is given by the line arrangement in \cref{fig: Line arrangement}.
\begin{figure}[h!]
  \centering
  \begin{minipage}[b]{0.4\textwidth}
  \hspace*{1cm}\scalebox{.75}{
\begin{tikzpicture}[line cap=round,line join=round,>=triangle 45,x=1cm,y=1cm]
\draw [line width=0.5pt] (-1.72,3.99)-- (2,0);
\draw [line width=0.5pt] (-1.72,3.99)-- (0,0);
\draw [line width=0.5pt] (0.72,3.99)-- (-2,0);
\draw [line width=0.5pt] (0.72,3.99)-- (2,0);
\draw [line width=0.5pt] (2.74,4.05)-- (0,0);
\draw [line width=0.5pt] (2.74,4.05)-- (-2,0);
\draw [shift={(0.23,1.905)},line width=0.5pt]  plot[domain=0.10427161192765513:3.245864265517448,variable=\t]({1*0.4323482392701513*cos(\t r)+0*0.4323482392701513*sin(\t r)},{0*0.4323482392701513*cos(\t r)+1*0.4323482392701513*sin(\t r)});
\draw [line width=0.5pt] (-2.3,4.006843575418994)-- (3,4.0586592178770955);
\draw [line width=0.5pt] (-2.3,0)-- (3,0);
\draw [line width=0.5pt] (-0.2,1.86)-- (-2.3,1.6425);
\draw [line width=0.5pt] (0.66,1.95)-- (3,2.1952095808383234);
\end{tikzpicture}}
    \caption{ }
    \label{fig: non-Pappus}
  \end{minipage}
  \hfill
  \begin{minipage}[b]{0.5\textwidth}
    \scalebox{0.6}{\begin{tikzpicture}[rotate=90, line cap=round,line join=round,>=triangle 45,x=1cm,y=1cm]
\clip (-3,3) rectangle + (7,-7);
		\node   (0) at (0, 6) {};
		\node   (1) at (-3, 6) {};
		\node   (2) at (3, 6) {};
		\node   (3) at (0, -6) {};
		\node   (4) at (1, -6) {};
		\node   (5) at (2, -6) {};
		\node   (6) at (3, -6) {};
		\node   (7) at (-1, -6) {};
		\node   (8) at (-2, -6) {};
		\node   (9) at (-3, -6) {};
		\node   (10) at (-4, -6) {};
		\node   (11) at (4, -6) {};

		\draw[line width=0.5pt] (1.center) -- (11.center);
		\draw[line width=0.5pt] (1.center) -- (6.center);
		\draw[line width=0.5pt] (1.center) -- (5.center);
		\draw[line width=0.5pt] (0.center) -- (3.center);
		\draw[line width=0.5pt] (0.center) -- (7.center);
		\draw[line width=0.5pt] (0.center) -- (4.center);
		\draw[line width=0.5pt] (2.center) -- (10.center);
		\draw[line width=0.5pt] (2.center) -- (9.center);
		\draw[line width=0.5pt] (2.center) -- (8.center);
\end{tikzpicture}}
    \caption{ }
    \label{fig: Line arrangement}
  \end{minipage}
\end{figure}
\end{example}
We now present results on the bounds of Chern numbers of simple matroids of rank 3. 
\begin{proposition}
\label{prop: ch positive}
Let M be a simple matroid of rank $3$ on the ground set $E = \{1,\ldots,n\}$. The Chern numbers are bounded by 
\begin{align*}
    0\leq&\overline{c}_1^2(M)\\
    0\leq&\overline{c}_2(M). 
\end{align*}
Left inequality holds if and only if $M$ has a coloop. 
\end{proposition}
\begin{proof}
Recall that the Chern number $\overline{c}_2(M)=\beta(M)$. Moreover,  the beta invariant $\beta({M})\geq 0$ and $\beta({M})= 0$ if and only if $M$ is disconnected or a loop, see \cite{crapo1967higher}. 
Finally, note that $M$ disconnected implies that $M$ has a coloop. 

Now, we prove by induction on $n$, that $\overline{c}_1^2(M)$ is positive if $M$ has no coloops. Assume that $n=4$ and recall that $t_m$ denotes the number of flats of rank $2$ of size $m$. Then $t_n=t_{n-1}=0$ implies $t_2=6$; so the Chern number of the matroid is  $\overline{c}_1^2(M)=1$. Assume now that $M$ is a matroid on a ground set $|E|=n\geq 5$, and let $i\in E$ be an element contained in at least three flats of rank 2. Note that such an element must exist by the assumption of no coloops. 

Let $\mathcal{F}_i=\{F_1\ldots F_t\}$ be the rank 2 flats of $M$ containing $i$, and denote by $M\setminus i$ the deletion matroid. We partition $\mathcal{F}_i$ as before, i.e., we let the set $\mathcal{F}_i=\mathcal{F}_1 \cup \mathcal{F}_2$, where $\mathcal{F}_1$ consists of flats $F$ such that $F\setminus i$ is a flat of rank 2 of $M\setminus i$, and $\mathcal{F}_2$ consists of flats $F$ of rank 2 such that $F\setminus i$ is a flat of rank 1 of $M\setminus i$. Recall also that the size of $\mathcal{F}_2$ is $|\mathcal{F}_2|=n-\sum_{F\in \mathcal{F}_1}(|F|-1)$. Then, the Chern number $\overline{c}_1^2(M)$ is given by
\begin{align*}
   \overline{c}_1^2(M)=\ & \overline{c}_1^2(M\setminus i)-5 - \sum_{F\in \mathcal{F}_1}(3(|F|-1)-4) + \sum_{F\in \mathcal{F}_1}(3|F|-4)\\
   \ &+ 2(n-\sum_{F\in \mathcal{F}_1}(|F|-1))\\
   =\ & \overline{c}_1^2(M\setminus i)-5 + \sum_{F\in \mathcal{F}_1} 1 + 2(\sum_{F\in \mathcal{F}_1} 1+ n - \sum_{F\in \mathcal{F}_1} (|F|-1))\\
   \geq\ & \overline{c}_1^2(M\setminus i)-5 +2t\\
   \geq\ & \overline{c}_1^2(M\setminus i)+1\\
   >\ & 1.
 \end{align*}
The last inequality follows from the induction hypothesis, i.e., that the Chern number $\overline{c}_1^2(M\setminus i)\geq 0$. 

Finally, if $M$ has a coloop, then $M$ has $t_{2} = n-1$ and $t_{n-1}=1$, whereas $t_k=0$ for all other $k$'s. By inserting these values for $c_1^2(M)$, we get that $c_1^2(M)=0.$
 \end{proof}
We also find an upper bound for the Chern numbers of simple matroids of rank 3.
\newbounds*
\begin{proof}
Left inequality is proven in \cref{prop: ch positive}. For the right inequality we begin with the first statement. Recall, by \cref{eq: c_1_squared}, that the Chern number $\overline{c}_1^2(M)$ of an arbitrary matroid is given by
\begin{align*}
    \overline{c}_1^2(M)=(3-n)^2-\sum_{m\geq 2} (2-m)^2t_m, 
\end{align*}
whereas the Chern number $\overline{c}_1^2(U_{3,n})$ of the uniform matroid $U_{3,n}$ is given by 
\begin{align*}
    \overline{c}_1^2(U_{3,n})=(3-n)^2.
\end{align*}
Then, since both the Chern number $\overline{c}_1^2(M)$ and the sum $\sum_{m\geq 2} (2-m)^2t_m$ of an arbitrary matroid are positive, we get the following inequality
$$\overline{c}_1^2(M)\leq (3-n)^2=\overline{c}_1^2(U_{3,n}).$$
The second statement follows from the inequality $\beta(M)\leq \beta(U_{d,n})$ when $M$ is of rank $\text{r}(M)=d$ on a ground set $|E|=n$. This is due to Brylawski, see Corollary 7.14 in \cite{brylawski1972decomposition}.
\end{proof}
We conjecture that a generalization of \cref{prop: new bounds} holds for a matroid of any rank. 
\begin{conj}
\label{conj: bounduni}
    Let $M$ be a simple matroid of rank $d+1$ on a ground set $E$ of size $n$. Then the Chern numbers of $M$ are bounded by 
$$|\bar{c}_1^{k_1}\cdots \bar{c}_d^{k_d}(M)| \leq |\bar{c}_1^{k_1}\cdots \bar{c}_d^{k_d}(U_{d+1,n})|$$
with equality if and only if $M=U_{3,n}$.
\end{conj}
\begin{remark}
In algebraic geometry, the  volume polynomial measures the top self-intersection number of divisors on a smooth projective variety. Eur introduces an analogous volume polynomial for matroids that agrees with the classical volume polynomial of the wonderful compactification when the matroid is representable \cite{eur2020divisors}. Moreover, he introduces a new invariant called the shifted rank-volume of a matroid $\text{shRVol}(M)$ as a particular specialization of its volume polynomial and proves that for representabe matroids $\text{shRVol}(M)\leq \text{shRVol}(U_{d,n})$ with equality if and only if $M=U_{d,n}$. Recall, that 
$\overline{c}_1^d(M)=w(\text{csm}_{d-1}(M))$ and that by the following isomorphism $MW_*(\Sigma_M)\simeq A^*(M)$ the CSM-cycle $\text{csm}_{d-1}(M)$ corresponds to a divisor $ch_1\in A^1(M)$. Recently, Rincón and Fife have proven a formula for the CSM-cycles of a matroid in the Chow ring of the matroid \cite{fife_2022}. We can use their formula to show that $ch_1(M)$ and Eur's Shifted rank divisor $D_M$, which is used to compute $\text{shRVol}(M)$, only differ by a term, see Proposition 5.4.8 in \cite{mannino2022chern}. Moreover, for a simple matroid $M$ of rank 3, we also prove that $\text{shRVol}(M)$ and $c_1^2(M)$ only differ by a function of $n$, see Lemma 6.4.2. in \cite{mannino2022chern}. 
\end{remark}      

We continue with our results on the quotient of the Chern numbers. The next result, generalizes Proposition 3.4 in \cite{eterovic2022geography} which gives bounds for the quotient of Chern numbers of line arrangements. Moreover, these bounds are the matroid analogue to the bound for complex projective surfaces.
\boundsc*
\begin{proof}
Proving the left inequality is equivalent to showing the inequality below
\begin{align}
    0&\leq (n-2)\overline{c}_1^2(M)-(2n-6)\overline{c}_2(M). \label{eq: inequality}
\end{align}
Recall from the proof of \cref{prop: Chern3} that the equality
\begin{align*}
    -n^2+n = \sum_{m\geq 2}(m-m^2)t_m
\end{align*}
holds. Then, by inserting for $\overline{c}_1^2(M)$, and $\overline{c}_2(M)$ in the right hand side of Inequality \ref{eq: inequality}
we get the following equality 
\begin{align*}
    (n-2)\overline{c}_1^2(M)-(2n-6)\overline{c}_2(M)&=-n^2+n+\sum_{m\geq 2}(mn-2n+2)t_m\\ &=\sum_{m\geq2}(-m^2+m)t_m + \sum_{m\geq 2}(mn-2n+2)t_m\\
    &=\sum_{m\geq2}(-m^2+m(1+n)+(2-2n))t_m.
\end{align*}
Moreover, the inequality $-m^2+m(1+n)+(2-2n)\geq 0$ holds for the values $2\leq m\leq n-1$ and the inequality $-m^2+m(1+n)+(2-2n)) > 0$ holds for all $3\leq m\leq n-2$. And since $m$ is less then $n-1$ by our assumption, Inequality \ref{eq: inequality} holds. The left equality in the theorem holds if and only if $M$ has $n(n-1)/2$ flats of rank 2 of size 2, hence if and only if $M$ is the uniform matroid $U_{3,n}$. 

Proving the right inequality is equivalent to showing the following inequality 
$$\overline{c}_1^2(M)-3\overline{c}_2(M)=n-\sum_{m\geq 2}t_m\leq 0.$$ The above inequality holds by Theorem 1 in \cite{dowling1975whitney} which states that for a finite geometric lattice of rank 3 the following inequality $W_1\leq W_2$ holds, where $W_1$ is the number of lattice elements of rank 1 and $W_2$ is the number of lattice elements of rank 2. 
Moreover, right equality holds if and only if the matroid arises from a finite projective plane $PG(2,q)$, see Theorem 1 in \cite{de1948p}.
\end{proof} 

The next theorem about the quotient of the Chern numbers is a generalization of Theorem 3.5 in \cite{eterovic2022geography} which is stated in terms of Chern numbers of real line arrangements.
\orientable*
\begin{proof} 
Note first that 
\begin{align*}
    5 \overline{c}_2(M)-2 \overline{c}_1^2(M) = -3 - \sum_{m\geq 2}(m-3)t_m.
\end{align*}
A pseudoline arrangement partitions $\mathbb{P}^2_\mathbb{R}$ into polygons and let $p_m$ be the number of $m$-gons. As in the proof of Theorem 3.5 in \cite{eterovic2022geography}, we refer to p.115 in \cite{hirzebruch1983arrangements} for the following statement. By the Euler characteristic formula, and by using that the Euler characteristic of $\mathbb{P}^2_\mathbb{R}$ is 1 one attains the following equality
 $$\sum_{m\geq 2}(m-3)p_m=-3-\sum_{m\geq2}(m-3)t_m.$$
 Since every two pseudolines intersect in exactly one point, and all the pseudolines do not intersect in the same point, the multiplicity $p_2=0$. Hence, the sum $\sum_{m\geq 2}(m-3)p_m\geq 0$ and the original inequality holds. Moreover, if the pseudoline arrangement $\mathcal{A}$ is simplicial, the multiplicity $p_k= 0$ for all $k\neq 3$. Hence, equality is achieved if and only if the matroid $M$ arises from a simplicial pseudoline arrangement. 
 
 In fact, the inequality $-3 - \sum_{m\geq 2}(m-3)t_m\geq 0$ for pseudoline arrangements was first proven by Melchior in 1940 and published in the unfavorable journal Deutsche Matematik \cite{AMS}, which we choose not to cite here.  
\end{proof}

In the table of \cref{fig: ch rank 3}, we have listed the values of the Chern numbers $\overline{c}_1^2(M)$, $\overline{c}_2(M)$ for some specific matroids and their corresponding quotient. Note that we have also included the non-Fano matroid, which arises from relaxing one of the flats of rank two of the Fano matroid, and the matroid arising from the Braid arrangement. In the graph of \cref{fig: ch rank 3}, we have plotted the Chern number pairs $(\overline{c}_2(M),\overline{c}_1^2(M))$. 
The slope of the lines represents the bound of the quotient of the Chern numbers given in \cref{thm: bounds ch 3} and \cref{thm: orientable}, which are respectively 3 and $\frac{5}{2}$.

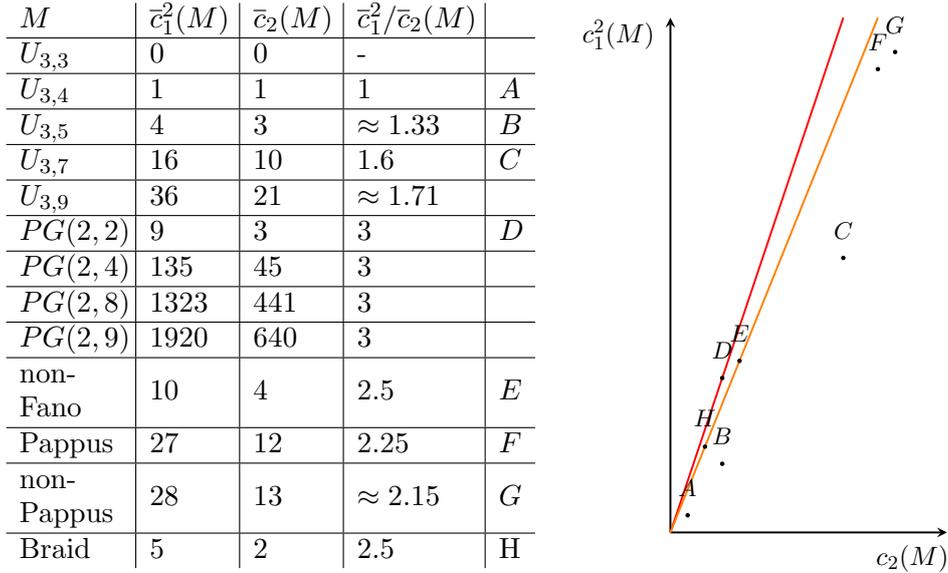
\begin{figure}[h!]
  \centering
  \begin{minipage}[b]{0.4\textwidth}
    \begin{tabular}{  m{3.5em} | m{1cm}| m{1cm}  | m{1.5cm} | m{0.3cm}} 
    $M$ & $\overline{c}_1^2(M)$ & $\overline{c}_2(M)$ & $\overline{c}_1^2/\overline{c}_2(M)$  &  \\
  \hline
  $U_{3,3}$ & 0 & 0 & - & \\ 
  \hline
  $U_{3,4}$ & 1 & 1 & 1 & $A$\\
  \hline
  $U_{3,5}$ & 4 & 3 & $\approx 1.33$ & $B$ \\
  \hline
  $U_{3,7}$ & 16 & 10 & 1.6 & $C$\\
  \hline
  $U_{3,9}$ & 36 & 21 & $\approx 1.71$ & \\
   \hline
  $PG(2,2)$ & 9 & 3 & 3 & $D$\\ 
  \hline
  $PG(2,4)$ & 135 & 45 & 3 & \\
  \hline
  $PG(2,8)$ & 1323 & 441 & 3 & \\
  \hline
  $PG(2,9)$ & 1920 & 640 & 3 & \\
  \hline
  non-Fano & 10 & 4 & 2.5 & $E$\\
  \hline
  Pappus & 27 & 12 & 2.25 & $F$\\
  \hline
  non-Pappus & 28 & 13 & $\approx 2.15$ & $G$\\
  \hline
  Braid & 5 & 2 & 2.5 & H\\
\end{tabular}
  \end{minipage}
  \hfill
  \raisebox{-23ex}{
  \begin{minipage}[b]{0.4\textwidth}
    \scalebox{0.91}
    {\begin{tikzpicture}[scale = 0.25]
\draw[line width=0.3mm, black, thick, -stealth] (0,0) -> (16,0);
\draw[line width=0.3mm, black, thick, -stealth,] (0,0) -> (0,30);
\draw[line width=0.3mm, red, thick, ] (0,0) -> (10,30);
\draw[line width=0.3mm, orange, thick, ] (0,0) -> (12,30);
\node[scale = 1] (1) at (-3,29) {$c_1^2(M)$};
\node[scale = 1] (2) at (14,-1.5) {$c_2(M)$};

\begin{scope}[every node/.style={circle,thick,draw,fill}, fill = black]
    \node[scale= 0.1, label ={\small $A$}] (A) at (1,1) {}; 
    \node[scale= 0.1, label ={\small $B$}] (B) at (3,4) {}; 
    \node[scale= 0.1, label ={\small $C$}] (C) at (10,16) {}; 

    \node[scale= 0.1, label ={\small $D$}] (D) at (3,9) {}; 
    \node[scale= 0.1, label ={\small $E$}] (E) at (4,10) {}; 
    \node[scale= 0.1, label ={\small $F$}] (F) at (12,27) {}; 
    \node[scale= 0.1, label ={\small $G$}] (G) at (13,28) {}; 
    \node[scale= 0.1, label ={\small $H$}] (H) at (2,5) {}; 
\end{scope}
\end{tikzpicture}}
  \end{minipage}
  }
  \caption{Chern numbers of matroids of rank 3.}
  \label{fig: ch rank 3}
\end{figure}



\bibliographystyle{alpha}
\bibliography{biblio}
\textit{Eline Mannino}\\
\textit{Department of Mathematics} \\
\textit{University of Oslo}\\
\textit{0316 Oslo}\\
\textit{Norway} \\
elinmann@uio.no

\end{document}